\documentclass[twocolumn]{article}

\usepackage{graphicx}          

\usepackage{amsfonts}
\usepackage{amsmath}
\usepackage{amssymb}
\usepackage{enumerate}

\newtheorem{lemma}{Lemma}[section]
\newtheorem{theorem}{Theorem}[section]
\newtheorem{definition}[theorem]{Definition}
\newtheorem{remark}[theorem]{Remark}

\newcommand{\norm}[1]{\left\lVert#1\right\rVert}
\newcommand{\abs}[1]{\left\vert#1\right\vert}
\DeclareMathOperator{\sgn}{sgn}

\usepackage{color}

\title{Mild solutions to the dynamic programming equation for stochastic optimal control problems} % 

\author{Viorel Barbu\thanks{A.I. Cuza University, Iasi, Romania} \and Chiara Benazzoli\thanks{Dept. of Mathematics, University of Trento, Italy} \and 
Luca Di Persio\thanks{Dept. of Computer Science, University of Verona, Italy}}

\begin{document}

\maketitle

$m$-accretive operator; Cauchy problem.

\begin{abstract}                          
We show via the nonlinear semigroup theory in $L^1(\mathbb{R})$ that the $1$-D dynamic programming equation associated with a stochastic optimal control problem with multiplicative noise has a unique mild solution $\varphi\in C([0,T];W^{1,\infty}(\mathbb{R}))$ with $\varphi_{xx}\in C([0,T];L^1(\mathbb{R}))$.  The $n$-dimensional case is also investigated. 
\end{abstract}

%\end{frontmatter}

\section{Introduction}
Consider the following stochastic optimal control  problem
\begin{equation}\label{min:pb}
\text{\underline{Minimize}}\quad\mathbb{E}\biggl\{\int_0^T\Bigl(g\bigl(X(t)\bigr)+h\bigl(u(t)\bigl)\Bigr)\,dt+g_0\bigl(X(T)\bigr)\biggr\},
\end{equation}
subject to $u\in\mathcal{U}$ and to state equation
\begin{equation}\label{SDE}
\begin{cases}dX=f(X)\,dt+\sqrt{u}\,\sigma(X)\,dW,\quad\text{for }t\in(0,T)\\
X(0)=X_0\end{cases}
\end{equation}
where  $\mathcal{U}$ is the set of all $\{\mathcal{F}_t\}_{t\ge 0}$-adapted processes $u:(0,T)\rightarrow\mathbb{R}^+=[0,+\infty]$ and $W:\mathbb{R}\rightarrow\mathbb{R}$ is an $1$-D Wiener process in a probability space $(\Omega,\mathcal{F},\mathbb{P})$, provided  the natural filtration $\{\mathcal{F}_t\}_{t\ge 0}$. Here $X_0\in\mathbb{R}$, while  $X:[0,T]\rightarrow\mathbb{R}$ is the strong solution to \eqref{SDE}. 

We would like to underline that the studied optimization problem is related to the so called stochastic volatility models, used in the financial framework, whose relevance has raised exponentially during last years. In fact such models, contrarily to the constant volatility ones as, e.g., the standard Black and Scholes approach, the Vasicek interest rate model, or the Cox-Ross-Rubistein model, allow to consider the more realistic situation of volatility levels changing in time. As an example, the latter is the case of the Heston model, see \cite{Heston}, where the variance is assumed to be a stochastic process following a Cox-Ingersoll-Ross (CIR) dynamic, see \cite{CIR} or \cite{CDP} and references therein for more recent related techniques, as well as the case of the Constant  Elasticity of Variance (CEV) model, see \cite{CEV}, where the volatility is expressed by a power of the underlying level, which is often referred as a local stochastic volatility model. Other interesting examples, which is  the object of our ongoing research particularly from the numerical point of view, include the Stochastic Alpha, Beta, Rho (SABR) model, see, e.g., \cite{Hagan}, and models which are used to estimate the stochastic volatility by exploiting directly  markets data, as  happens using the GARCH approach and its variants.  

Within latter frameworks and due to several macroeconomic crises that have affected different (type of) financial markets worldwide, governments decided to become {\it active players of  the game}, as, e.g., in the recent case of the {\it Volatility Control Mechanism} (VCM) established   for  the securities, resp. for the derivatives, market established in August 2016, resp. in January 2017, within the Hong Kong Stock Exchange (HKEX) framework, see, e.g., \cite{Stein1,Stein2} and references therein for other applications and examples.

\noindent
\textbf{Hypotheses}:
\begin{enumerate}
\item $h:\mathbb{R}\rightarrow\mathbb{R}$ is convex, continuous and $h(u)\ge\alpha_1\,|u|^2+\alpha_2$, $\forall u\in\mathbb{R}$, for some $\alpha_1>0, \alpha_2\geq 0$.
\item $f\in C_b^2(\mathbb{R})$, $f''\in L^1(\mathbb{R})$, $g,g_0\in W^{2,\infty}(\mathbb{R})$.
\item $\sigma\in C_b^1(\mathbb{R})$, and
\begin{equation}\label{eq:4}
|\sigma(x)|\ge\rho>0,\quad\forall x\in\mathbb{R}.
\end{equation}
\end{enumerate}
We set
\[
H(u)=h(u)+I_{[0,\infty)}(u)=\begin{cases}h(u)\quad\text{if }u\ge0\\+\infty\quad\text{otherwise}\end{cases}
\]
and we denote by $H^*$ the Legendre conjugate of $H$, namely,
\begin{equation}\label{eq:5}
H^*(p)=\sup\{p\,u-H(u)\,:\;u\in\mathbb{R}\},\quad\forall p\in\mathbb{R}.
\end{equation}

We have $(H^*)'(p) = ( \partial h  + N_{[0,\infty)})^{-1}p \in Lip(\mathbb{R})$, where $\delta h$ is the subdiffential of $h$, and $N_{[0,\infty)}$ is the normal cone to $[0,\infty)$. This yields 
\begin{equation}\label{eq:5prime}
 (H^*)'' \in L^\infty(\mathbb{R}) \;,\; 0 \leq (H^*)' (p) \leq C ( \abs{ p} +1) \;,\; \forall p \in \mathbb{R}\;.
\end{equation}

We denote also by $j$ the potential of $H^*$, that is
\[
j(r)=\int_0^rH^*(p)\,dp,\quad\forall r\in\mathbb{R}.
\]

The dynamic programming equation corresponding to the stochastic optimal control problem \eqref{min:pb} is given by 
(see, e.g., \cite{fleming},\cite{oksendal}),
\begin{equation}\label{DPE}
\begin{cases}\varphi_t(t,x)+\min_{u}\bigl\{\frac{1}{2}\sigma^2\,\varphi_{xx}(t,x)\,u+H(u)\bigr\}\\
\hspace{1cm}+f(x)\,\varphi_x(t,x)+g(x)=0,\quad\forall t\in[0,T],x\in\mathbb{R}\\
\varphi(T,x)=g_0(x),\quad x\in\mathbb{R},\end{cases}
\end{equation}
or equivalently
\begin{equation}
\begin{cases}\label{DPE:conj}
\varphi_t(t,x)-H^*\bigl(-\frac{1}{2}\sigma^2\,\varphi_{xx}(t,x)\bigr)+f(x)\,\varphi_x(t,x)\\
\hspace{2.5cm}+g(x)=0,\quad\forall(t,x)\in[0,T]\times\mathbb{R}\\
\varphi(T,x)=g_0(x),\quad x\in\mathbb{R}\,.
\end{cases}
\end{equation}
Moreover, if $\varphi$ is a smooth solution to \eqref{DPE} the associated feedback controller
\begin{equation}\label{feedback}
u(t)=\arg\min_{u}\Bigl\{\frac{1}{2}\,\sigma^2\,\varphi_{xx}\bigl(t,X(t)\bigr)\,u+H(u)\Bigr\}\,,
\end{equation}
is optimal for problem \eqref{min:pb}.

Up to our knowledge, in literature the rigorous treatment of existence theory for equation \eqref{DPE} has been shown, so far within the theory of viscosity solutions only. (See, e.g., \cite{crandall}.)
Here we shall exploit a different approach, namely  we use a suitable transformation aiming at reducing  \eqref{DPE} to an one dimensional Fokker-Planck equation which is then treated as a nonlinear Cauchy problem in $L^1(\mathbb{R})$.
The $n$-dimensional case is also studied in section~\ref{sec:nd}. As regards the non-degenerate hypothesis (3) it will be later on dispensed by assuming more regularity on function $\sigma$. (See section 4 below.)

\subsection{Notation and basic results}
We shall use the standard notation for functional spaces on $\mathbb{R}$. In particular $C^k_b(\mathbb{R})$ is the space of functions $y:\mathbb{R}\rightarrow\mathbb{R}$, differentiable of order $k$ and with bounded derivatives until order $k$. By $L^p(\mathbb{R})$, $1\le p\le\infty$, we denote the classical space of Lebesgue-measurable $p$-integrable functions on $\mathbb{R}$ with the norm $\norm{ \cdot }_p$ and by $H^k(\mathbb{R}^n)$, $W^{k,p}(\mathbb{R}^n)$, $k=1,2$, the standard Sobolev spaces on $\mathbb{R}^n$, $n=1,2$.
We set also $y_x=y'=\partial y/\partial x$, $y_t=\partial y/\partial t$, $y_{xx}=\partial^2y/\partial x^2$, for $x\in\mathbb{R}$ and $\Delta y(x)=\sum_{i=1}^n\frac{\partial^2 y}{\partial x_i^2}$, for $x\in\mathbb{R}^n$.
By $\mathcal{D}'(\mathbb{R}^n)$ we denote the space of  Schwartz distributions on $\mathbb{R}^n$.

\begin{definition}[Accretive operator]
Given a Banach space $X$, a nonlinear operator $A$ from $X$ to itself, with domain $D(A)$, is said to be \emph{accretive} if $\forall u_i\in D(A),\forall v_i\in A\,u_i$, $i=1,2$, there exists $\eta\in J(u_1-u_2)$ such that
\begin{equation}
_X\langle v_1-v_2,\eta\rangle_{X'}\ge0\, ,
\end{equation}
where $X'$ is the dual space of $X$, $_X\langle\cdot,\cdot\rangle_{X'}$ is the {duality pairing} and $J:X\rightarrow X'$ is the {\it duality mapping} of $X$. (See, e.g., \cite{B:1}.)
\\
An accretive operator $A$ is said to be \emph{$m$-accretive} if $\mathbb{R}(\lambda\,I+A)=X$ for all (equivalently some) $\lambda>0$, while it
is said to be $quasi-$m$-accretive$ if there is $\lambda_0\in\mathbb{R}$ such that $\lambda_0\,I+A$ is $m$-accretive.
\end{definition}
We refer to \cite{B:1} for basic results on $m$-accretive operators in Banach spaces and the corresponding associated Cauchy problem.

\section{Existence results}
We set
\begin{equation}\label{2:1}
y(t,x)=-\varphi_{xx}(T-t,x),\quad\forall t\in[0,T],x\in\mathbb{R},
\end{equation}
and we rewrite eq. \eqref{DPE:conj} as
\begin{equation}\label{DPE:conj:1}
\begin{cases}
y_t(t,x)-\Bigl(H^*\bigl(\frac{\sigma^2}{2}\,y(t,x)\bigr)\Bigr)_{xx}+f''(x)\varphi_x(T-t,x) \\
\hspace{0.5cm}-2f'(x)y(t,x) -f(x)y_x(t,x)=-g''(x),\\
\hspace{5.5cm}\text{in }(0,T)\times\mathbb{R}\\
y(0,x)=-g''_0(x),\quad x\in\mathbb{R}.
\end{cases}
\end{equation}

We recall (see \cite{benilan} for details), that, for $z\in L^1(\mathbb{R})$, the equation
\begin{equation}\label{eq:1}
-\Psi''=z,\quad\text{in }\mathcal{D}'(\mathbb{R})\,,
\end{equation}
has a unique solution $\Psi=\Phi(z)\in W^{1,\infty}(\mathbb{R})$ and $\| \Psi \|_{W^{1,\infty}(\mathbb{R})} \leq C \| z \|_{1}$.
Then by \eqref{2:1} we have
\begin{equation}\label{2:4}
\varphi(t,x)=-\Phi\bigl(y(T-t,x)\bigr)\in W^{1,\infty}(\mathbb{R}),\quad\forall t\in[0,T].
\end{equation}
Setting
\begin{equation}\label{2:6}
B\,y=-f''(\Phi(y))'-2\,f'y,\quad\forall y\in L^1(\mathbb{R})\,,
\end{equation}
and taking into account that $f'\in L^\infty(\mathbb{R})$, $f'' \in L^1(\mathbb{R})$, and $\| (\Phi(y))' \|_{\infty} \leq \|\Phi\|_{W^{1,\infty}(\mathbb{R})}\leq C\| y\|_{1}$, we obtain for operator $B$ the estimate
\begin{equation}
||B\,y||_{1}\le C\,||y||_{1},\quad\forall y\in L^1(\mathbb{R})\,.
\end{equation}
Therefore eq. \eqref{DPE:conj:1} can be rewritten as follows
\begin{equation}\label{eq:2}
\begin{cases}
y_t-\Bigl(H^*\bigl(\frac{\sigma^2}{2}\,y\bigr)\Bigr)_{xx}-f\,y_x+B\,y=g_1,\,\text{in }[0,T]\times\mathbb{R}\\
y(0)=y_0\in L^1(\mathbb{R})\,,
\end{cases}
\end{equation}
where $y_0=-g_0''$ and $g_1=-g''$ in $\mathcal{D}'(\mathbb{R})$.

\begin{definition}\label{mild}
The function $y\colon[0,T]\times\mathbb{R}\to\mathbb{R}$ is said to be a {\it mild} solution to equation \eqref{eq:2} if $y\in C([0,T];L^1(\mathbb{R}))$ and
\begin{equation}
y(t)=\lim_{\epsilon\rightarrow 0}y_\epsilon(t)\text{ in }L^1(\mathbb{R}),\quad\forall t\in[0,T]\;,
\end{equation}
\begin{equation}
y_\epsilon(t)=y_\epsilon^i,\text{ for }t\in[i\,\epsilon,(i+1)\,\epsilon],\,i=0,1,\dots,N=\Bigl[\frac{T}{\epsilon}\Bigr]\,,
\end{equation}
\begin{multline}\label{eq:3}
\frac{1}{\epsilon}\,(y_\epsilon^{i+1}-y_\epsilon^i)-\Bigl(H^*\bigl(\frac{\sigma^2}{2}\,y_\epsilon^{i+1}\bigr)\Bigr)''\\-f(y_\epsilon^{i+1})'+B\,y_\epsilon^{i+1}=g_1,\quad\text{in }\mathcal{D}'(\mathbb{R}),
\end{multline}
\[
y_\epsilon^0=y_0,\,y_\epsilon^i\in L^1(\mathbb{R}),\,i=0,1,\dots,N\,.
\]
\end{definition}

We have
\begin{theorem}\label{mild:sol}
Under hypotheses (1)-(3) eq. \eqref{DPE:conj:1} has a unique mild solution $y$.
Assume further that $j(\frac{\sigma^2}{2}\,y_0)\in L^1(\mathbb{R})$. Then $j(\frac{\sigma^2}{2}\,y_\epsilon)\in L^\infty([0,T];L^1(\mathbb{\mathbb{R}}))$ and $\left(H^*(\frac{\sigma^2}{2}\,y)\right)_x\in L^2([0,T]\times\mathbb{R})$. 
\end{theorem}
Theorem \ref{mild:sol} will be proven by using the standard existence theory for the Cauchy problem in Banach spaces with nonlinear quasi-$m$-accritive operators.
Now taking into account that for $y\in C([0,T];L^1(\mathbb{R}))$ equation \eqref{eq:1} uniquely defines the function $\varphi\in C([0,T];W^{1,\infty}(\mathbb{R}))$, by Theorem \ref{mild:sol} we obtain the following existence result for the dynamic programming equation \eqref{DPE}.

\begin{theorem}\label{thm:2}
Under hypothesis (1)-(3) there is a unique {\it mild} solution
\begin{equation}\label{2:11}
%\begin{gathered}
\varphi\in C\bigl([0,T];W^{1,\infty}(\mathbb{R})\bigr) \;,\;
\varphi''\in C\bigl([0,T];L^1(\mathbb{R})\bigr)\,,
%\end{gathered}
\end{equation}
to equation \eqref{DPE}. Moreover, if $h(\lambda u)\le C_\lambda h(u)$ $\forall u\in\mathbb{R}$, $\lambda>0$ and $j(-\frac{\sigma^2}{2}\,g''_0)\in L^1(\mathbb{R})$, then $H^*\bigl(-\frac{\sigma^2}{2}\,\varphi_{xx}(T-t,x)\bigr)\in L^2([0,T]\times\mathbb{R})$.
\end{theorem}

According to the Definition \ref{mild} and \eqref{2:4}, by {\it mild} solution $\varphi$ to equation \eqref{DPE}, we mean a function
$\varphi\in C([0,T];W^{1,\infty}(\mathbb{R}))$ defined by
\begin{equation}\label{2:12}
\varphi(t)=\lim_{\epsilon\rightarrow0}\varphi_\epsilon(t)\text{ in }W^{1,\infty}(\mathbb{R}),\quad\forall t\in[0,T] \,,
\end{equation}
\begin{equation}\label{2:13}
\varphi_\epsilon(t)=\Psi(y_\epsilon^i),\quad t\in[T-(i+1)\,\epsilon,T-i\,\epsilon],\,
\end{equation}
for $i=0,1,\dots,N=\left[\frac{T}{\epsilon}\right]$ and $\{y_\epsilon^i\}$ is the solution to \eqref{eq:3}.

In particular, the {\it mild} solution $\varphi$ to equation \eqref{DPE} is in $H_{\text{loc}}^2(\mathbb{R})\cap W^{1,\infty}(\mathbb{R})$. Therefore, the feedback controller \eqref{feedback} is well defined on $[0,T]$.

\begin{remark}
The principal advantage of Theorem \ref{mild:sol} compared with standard existence results expressed in terms of viscosity solutions is the regularity of $\varphi$ and the fact that the optimal feedback controller can be computed explicitly by the finite difference scheme \eqref{2:12}-\eqref{2:13}. This will be treated in a forthcoming paper.
\end{remark}

\section{Proof of Theorem \ref{mild:sol}}
The idea is to write equation \eqref{eq:2} as a Cauchy problem of the form
\begin{equation}\label{3:1}
\begin{cases}\frac{dy}{dt}+A\,y+B\,y=g_1,\quad\text{in }[0,T]\\
y(0)=y_0
\end{cases}\,,
\end{equation}
in the space $L^1(\mathbb{R})$, where $A$ is a suitable nonlinear quasi-$m$-accretive operator. The operator $A:D(A)\subset L^1(\mathbb{R})\rightarrow L^1(\mathbb{R})$ is defined as follows
\begin{equation}
A\,y=-\Bigl(H^*\bigl(\frac{\sigma^2}{2}\,y\bigr)\Bigr)''-f\,y'\quad\text{ in }\mathcal{D}'(\mathbb{R}),\,\forall y\in D(A)\,,
\end{equation}
\begin{equation*}
\begin{aligned}
D(A)=\Bigl\{y\in L^1(\mathbb{R}):H^*\bigl(\frac{\sigma^2\,y}{2}\bigr)\in L^\infty(\mathbb{R}),&\\Ay\in L^1&(\mathbb{R})\Bigr\}\,,
\end{aligned}
\end{equation*}
where the derivatives are taken in the $\mathcal{D}'(\mathbb{R})$ sense.
\begin{lemma}\label{lemma3:1}
For each $\eta\in L^1(\mathbb{R})$ and $\lambda\ge\lambda_0=||f'||_{\infty}$ there exists a unique solution $y=y(\eta)$ to equation
\begin{equation}\label{3:3}
\lambda\,y+A\,y=\eta.
\end{equation}
Moreover, it holds
\begin{equation}\label{3:4}
||y(\eta)-y(\bar{\eta})||_{1}\le(\lambda-\lambda_0)^{-1}\,||\eta-\bar{\eta}||_{1}\,,
\end{equation}
$\forall\eta,\bar{\eta}\in L^1(\mathbb{R}),\lambda>\lambda_0$,
hence $A$ turns to be quasi-$m$-accretive in $L^1(\mathbb{R})$.
\end{lemma}
\begin{proof}[Proof of Lemma \ref{lemma3:1}]
Assume first that $\eta\in L^1(\mathbb{R})\cap L^2(\mathbb{R})$.
For each $\nu>0$ consider the equation
\begin{equation}\label{eq:9}
\lambda\,y-\nu\,y''-\Bigl(H^*\bigl(\frac{\sigma^2}{2}\,y\bigr)\Bigr)''+\nu\,H^*\bigl(\frac{\sigma^2}{2}\,y\bigr)-f\,y'=\eta,
\end{equation}
in $\mathcal{D}'(\mathbb{R})$. Equivalently,
\begin{multline}\label{eq:8}
(\lambda-\nu^2)\Bigl(\nu\,I-\frac{d^2}{dx^2}\Bigr)^{-1}\,y+H^*(\frac{\sigma^2}{2}\,y)+\nu\,y\\-\Bigl(\nu\,I-\frac{d^2}{dx^2}\Bigr)^{-1}\,(f\,y')=\Bigl(\nu\,I-\frac{d^2}{dx^2}\Bigr)^{-1}\,\eta\,,
\end{multline}
where  $z=\bigl(\nu\,I-\frac{d^2}{dx^2}\bigr)^{-1}\,y$ is defined by equation
\begin{equation}\label{eq:6}
\nu\,z-z''=y,\quad\text{in }\mathcal{D}'(\mathbb{R})\,.
\end{equation}
Note that by Hypothesis (2) the operator $\Gamma\,y=(\lambda-\nu^2)\bigl(\nu I-\frac{d^2}{dx^2}\bigr)^{-1}y - \bigl(\nu I-\frac{d^2}{dx^2}\bigr)^{-1}(fy')+\nu y$ is linear continuous in $L^2(\mathbb{R})$ and by \eqref{eq:6} we have that
\begin{equation}
\langle z,y\rangle_2=\nu\,||z||^2_2+||z'||_2^2\,,
\end{equation}
\begin{equation}
\begin{aligned}
-\Bigl\langle\Bigl(\nu\,I-\frac{d^2}{dx^2}\Bigr)^{-1}\,(f\,y'),\,y\Bigr\rangle_2&=-\langle f\,y',z\rangle_2\\
&=  \langle y,f'\,z+f\,z'\rangle_2\\
&\le||f'||_\infty\,||y||_2\,||z||_2\\
&\hspace{0.5cm}+||f||_\infty\,||y||_2\,||z'||_2.\label{eq:7}
\end{aligned}
\end{equation}
Here $||\cdot||_2$ and $\langle\cdot,\cdot\rangle_2$ are the norm and the scalar product in $L^2(\mathbb{R})$, respectively, and by $||\cdot||_p$, $1\le p\le \infty$ we denote the norm of $L^p(\mathbb{R})$.
We note that  Hypothesis (1) and \eqref{eq:5} imply that the function $H^*$ is continuous, monotonically non--decreasing, and
\begin{equation}
C_1\le H^*(v)\le C_2\,v^2,\quad\forall v\in\mathbb{R}\,.
\end{equation}
Furthermore, by \eqref{eq:6}-\eqref{eq:7}, we have
\begin{align*}
\langle\Gamma\,y,y\rangle_2&=\nu\norm{y}^2_2 + (\lambda-\nu^2) \langle y,z\rangle_2-\langle f\,y',z\rangle_2\\
&\ge \nu\norm{y}_2^2+(\lambda-\nu^2)(\nu\norm{z}^2_2+\norm{z'}^2_2) \\
&\hspace{1cm}- \norm{y}_2(\norm{f'}_\infty \norm{z}_2+\norm{f}_\infty\norm{z'}_2)\\
&\ge \nu\norm{y}_2^2+(\lambda-\nu^2)(\nu\norm{z}^2_2+\norm{z'}^2_2) \\
&\hspace{1cm}- C(f) \norm{y}_2( \norm{z}_2+\norm{z'}_2)
\end{align*}
The latter yields
\begin{equation}\label{lambda-inequality}
\langle\Gamma\,y,y\rangle_2\ge\frac{\nu}{2}\,||y||_2^2,\quad\lambda\ge C\left(\frac{1}{\nu}+\nu^2\right),\forall\nu>0\,,
\end{equation}
where $C$ is dependent on $\nu$.
By assumption (3) we have that the operator $y\rightarrow\mathcal{H}(y)\equiv H^*\bigl(\frac{\sigma^2}{2}\,y\bigr)$ is maximal monotone in $L^2(\mathbb{R})$, hence, by \eqref{lambda-inequality}, $\Gamma$ is maximal monotone and coercive, i.e. positively definite, therefore we have	 
\[
\mathbb{R}(\Gamma+\mathcal{H})=L^2(\mathbb{R})\,,
\]
for $\lambda\ge\lambda^*= C(\frac{1}{\nu}+\nu^2)$. Consequently, for each $\nu>0$ and $\lambda\ge\lambda^*$, eq. \eqref{eq:8} (equivalently eq. \eqref{eq:9}) has a unique solution $y=y_{\lambda,\nu}\in L^2(\mathbb{R})$, with $H^*\bigl(\frac{\sigma^2}{2}\,y_{\lambda,\nu}\bigr)\in L^2(\mathbb{R})$.

We  have also
\[
z_{\lambda,\nu}+z_{\lambda,\nu}''\in L^2(\mathbb{R})\,,
\]
so that $z_{\lambda,\nu}\in H^2(\mathbb{R})$. 

Since by assumption (3) the operator $z\rightarrow \nu z+H^*\bigl(\frac{\sigma^2}{2}\,z\bigr)$ is invertible in $L^2(\mathbb{R})$, and its inverse maps inverse $H^1(\mathbb{R})$ into itself, we infer that $y_{\lambda,\nu}\in H^1(\mathbb{R})$.

It is worth to mention that by \eqref{eq:9}, we have
\begin{multline*}
\lambda\norm{y_{\lambda,\nu}(\eta)-y_{\lambda,\nu}(\bar \eta)}_1 \le\\ \norm{f'}_\infty\norm{y_{\lambda,\nu}(\eta)-y_{\lambda,\nu}(\bar \eta)}_1+\norm{\eta-\bar\eta}_1
\end{multline*}
$\forall\eta,\bar\eta\in L^1(\mathbb{R})$, so that  
\begin{equation}\label{3:12}
\norm{y_{\lambda,\nu}(\eta)-y_{\lambda,\nu}(\bar \eta)}_1\le\frac{1}{\lambda-\lambda_0}\norm{\eta-\bar\eta}_1\,,
\end{equation}
$\forall\eta,\bar\eta\in L^1(\mathbb{R})$, for $\lambda\ge\max\left(\lambda_0,\lambda^*\right)$ and where $\lambda_0=\norm{f'}_\infty$.
To get  \eqref{3:12}, we simply multiply the equation
\begin{multline*}
\lambda(y_{\lambda,\nu}(\eta)-y_{\lambda,\nu}(\bar \eta))-\nu (y_{\lambda,\nu}(\eta)-y_{\lambda,\nu}(\bar \eta))''\\
+\nu\left(H^*\left(\frac{\sigma^2}{2}y_{\lambda,\nu}(\eta)\right)-H^*\left(\frac{\sigma^2}{2}y_{\lambda,\nu}(\bar\eta)\right)\right)\\
-\left(H^*\left(\frac{\sigma^2}{2}y_{\lambda,\nu}(\eta)\right)-H^*\left(\frac{\sigma^2}{2}y_{\lambda,\nu}(\bar\eta)\right)\right)''\\
+f(y_{\lambda,\nu}(\eta)-y_{\lambda,\nu}(\bar \eta))'= \eta-\bar \eta
\end{multline*}
by $\zeta\in L^\infty(\mathbb{R})$
\begin{multline*}\zeta\in\sgn(y_{\lambda,\nu}(\eta)-y_{\lambda,\nu}(\bar \eta)) =\\=\sgn\left(H^*\left(\frac{\sigma^2}{2}y_{\lambda,\nu}(\eta)\right)-H^*\left(\frac{\sigma^2}{2}y_{\lambda,\nu}(\bar\eta)\right)\right)\,,\end{multline*} 
where $\sgn r = \frac{r}{\mid r \mid}$ for $r \neq 0$, $\sgn 0 = [-1,1]$ and
we integrate on $\mathbb{R}$, taking into account that
\[
-\int_\mathbb{R} y''\sgn y dx\ge0\,,\quad\forall y\in H^1(\mathbb{R})\,,
\]
\[
\int_{\mathbb{R}} fy'\sgn y dx=\int_{\mathbb{R}} f\abs{y}'dx=-\int_\mathbb{R} f'\abs{y}dy\,.
\]
For a rigorous proof of these relations we replace $\sgn y$ by $X_\delta(y)$, where $X_\delta$ is a smooth approximation of signum function, while  $\delta\to 0$, see , e.g., \cite{B:1}, p. 115. If $\eta\in L^1(\mathbb{R})$ and $\{\eta_n\}_{n=1}^\infty\subset L^1(\mathbb{R})\cap L^2(\mathbb{R})$ is strongly convergent to $\eta\in L^1(\mathbb{R})$, we can proceed as above to obtain for the corresponding solution $y_n$ to \eqref{eq:9} the estimate \eqref{3:12}, namely,
\[
\norm{y_n-y_m}_1\le(\lambda-\lambda_0)^{-1}\norm{\eta_n-\eta_m}_1\,,\,\forall \lambda>\max\left(\lambda^*,\lambda_0\right).
\]
Hence there exists $y\in L^1(\mathbb{R})$ such that
\begin{equation}\label{3:13}
y_n\to y \quad\text{in $L^1(\mathbb{R})$ as $n\to\infty$}\,.
\end{equation}
By \eqref{eq:8}, we have
\begin{equation}\label{3:14}
\begin{aligned}
&(\lambda-\nu^2)\left(\nu I - \frac{d^2}{dx^2}\right)^{-1}y_n+H^*\left(\frac{\sigma^2}{2}y_n\right)+\nu y_n\\
&\hspace{1cm}-\left(\nu I - \frac{d^2}{dx^2}\right)^{-1}(fy'_n)= \left(\nu I - \frac{d^2}{dx^2}\right)^{-1}\eta_n\,.
\end{aligned}
\end{equation}

By \eqref{eq:1} and \eqref{eq:6} , we have
\begin{equation}\label{eq:36}
\| z_n \|_{W^{1,\infty}(\mathbb{R})} \leq \| \nu z_n-y_n\|_{1}\leq (\nu+1) \| y_n \|_{1}\:. 
\end{equation}

Let $\theta_n := \left( \nu I - \frac{d^2}{dx^2} \right)^{-1} \left(f y_n'\right)$, 

that is $\nu \theta_n - \theta_n^{''} = f y_n' = \left( f y_n \right)' - f' y_n$ 
in $\mathcal{D}' \left( \mathbb{R}^n\right)$. Equivalently
\begin{equation} \label{eq:37prime}
\begin{split}
& \nu\left(\theta_n(x) + \int_0^x f y_n d\xi\right) - \left( \theta_n(x) + \int_0^x f y_n d\xi\right)^{''}=\\ 
&=\nu \int_0^x f y_n d\xi -f' y_n \:.
\end{split}
 \end{equation}

 This yields 
\begin{equation} 
\begin{split}
&\norm{ \nu  \theta_n + \nu \int_0^x f y_n d\xi }_{1} \leq \nu \norm{ \int_0^x f y_n d\xi }_{1}
+ \| f' y_n \|_{L^1} \\
& \hspace{3cm}\leq \nu \| f \|_{\infty} \|y_n\|_{1} + \| f' \|_{\infty} \|y_n\|_{1}
\end{split}
 \end{equation}
and then
$$ \nu \| \theta_n \|_{1} \leq \left( (\nu +1) \|f\|_{\infty}  \| y_n \|_1+\| f' \|_{\infty}\right) \|y_n \|_{1}\,.$$

On the other hand, by \eqref{eq:37prime}, we have
\begin{equation}
 \begin{split}
    & \norm{ \theta_n +\int_0^x f y_n d\xi }_{W^{1,\infty}(\mathbb{R})} \leq  
		\norm{\nu \theta_n + \nu \int_0^x f y_n d\xi }_{1} \\
		&+ \| \nu \int_0^x f y_n d\xi -f' y_n \|_{1} \leq \nu \| \theta \|_{1}
		\\
		&+\left( 2 \nu \| f \|_{\infty}+ \| f'\|_{\infty} \right)\|y\|_{1}\:.
 \end{split}
\end{equation}
Hence
$$
 \| \theta_n \|_{W^{1,\infty}(\mathbb{R})} \leq \left( (3\nu+1) \|f\|_{\infty} 
+ 2\|f'\|_{\infty} \right) \|y_n\|_{1}\:.
$$
This yields
\begin{equation}\label{eq:3:15}
\norm{\left(\nu I - \frac{d^2}{dx^2}\right)^{-1}(fy'_n)}_\infty \le C \norm{y_n}_1\le \frac{C_1}{\lambda-\lambda_0}\norm{\eta_n}_1
\end{equation}
and therefore, by  \eqref{3:14}, we derive the estimate
\[
\norm{H^*\left(\frac{\sigma^2}{2}y_n\right)+\nu y_n}_\infty \le C \norm{y_n}_1\le \frac{C_1}{\lambda-\lambda_0}\norm{\eta_n}_1\,.
\]
Since, by hypothesis (1) $H^*(v)v \geq 0, \forall v \in \mathbb{R}$, the latter implies that
\begin{equation}\label{3:17}
\norm{H^*\left(\frac{\sigma^2}{2}y_n\right)}_\infty+\nu\norm{ y_n}_\infty \le \frac{C_1}{\lambda-\lambda_0}\norm{\eta_n}_1\,,\quad\forall n \,,
\end{equation}
where $C_1$ is still independent of $n$ as well as on $\nu$.

By \eqref{3:13} and \eqref{3:17}, it follows that 
\begin{equation}\label{3:18}
H^*\left(\frac{\sigma^2}{2}y_n\right) \stackrel{n\to\infty}{\longrightarrow} H^*\left(\frac{\sigma^2}{2}y\right)\,,
\end{equation}
strongly in $L^1(\mathbb{R})$, and therefore $y=y_{\lambda,\nu}\in L^\infty(\mathbb{R})\cap L^1(\mathbb{R})$ solves \eqref{eq:9}. Furthermore, by \eqref{3:12} and \eqref{3:17}, we have 
\begin{equation}\label{3:19}
\norm{y_{\lambda,\nu}}_1+\norm{H^*\left(\frac{\sigma^2}{2}y_{\lambda,\nu}\right)}_\infty+\nu\norm{ y_{\lambda,\eta}}_\infty \le \frac{C_1}{\lambda-\lambda_0}\norm{\eta}_1
\end{equation}
$\forall \lambda>\max\left(\lambda^*,\lambda_0\right)$, where $C_1$ is independent of $\nu$.
We also obtain that inequality \eqref{3:12} holds for solution $y_{\lambda,\nu}$ to \eqref{eq:9}, with $\eta\in L^1(\mathbb{R})$ only. 
Now we are going to extend the solution  $y_{\lambda,\nu}$ to \eqref{eq:9}  for all $\lambda>\lambda_0$. To this end we set $G^\nu_\lambda=\Gamma+\mathcal{H}$, rewriting  \eqref{eq:9} as follows $G^\nu_\lambda=\eta$.
For every $\lambda>0$, we can equivalently write this as
\begin{equation}\label{3:21}
y=(G^\nu_{\lambda+\delta})^{-1}(\eta)+\delta(G^\nu_{\lambda+\delta})^{-1}(\eta)\,.
\end{equation}
By \eqref{3:12} we also have
\[
\norm{(G^\nu_{\lambda+\delta})^{-1}}_{L(L^1(\mathbb{R}),L^1(\mathbb{R}))}\le\frac{1}{\lambda-\lambda_0}\,,
\]
then, by contraction principle,  \eqref{3:21} has a unique solution $y=y_{\lambda,\nu}\in L^1(\mathbb{R})$, for all $\lambda>\lambda_0$. Estimate \eqref{3:19} extends for all $\lambda>\lambda_0$.
In order to complete the proof of Lemma~\ref{lemma3:1}, we are going to let $\nu\to0$ in equation \eqref{eq:9}, or, more precisely, in \eqref{eq:8} which holds for all $\lambda>\lambda_0$.
As noted before, for all $z\in L^1(\mathbb{R})$, we have
\begin{equation*}
\norm{\left(\nu I - \frac{d^2}{dx^2}\right)^{-1} z}_{W^{1,\infty}(\mathbb{R})}\le C\norm{z}_1
\end{equation*}
and
\begin{equation}\label{3:22}
\lim_{\nu\to0} \left(\nu I - \frac{d^2}{dx^2}\right)^{-1} z =\left(- \frac{d^2}{dx^2}\right)^{-1} z \text{ in $W^{1,\infty}(\mathbb{R})$}\,,
\end{equation}
consequently
\begin{equation*}
\norm{\left(\nu I - \frac{d^2}{dx^2}\right)^{-1} (fz')}_\infty\le C\norm{z}_1
\end{equation*}
and
\begin{multline*}
\lim_{\nu\to0} \left(\nu I - \frac{d^2}{dx^2}\right)^{-1} (fz') =\lim_{\nu\to0} \frac{d}{dx}\left(\nu I - \frac{d^2}{dx^2}\right)^{-1} (fz)\\ + \lim_{\nu\to0} \left(\nu I - \frac{d^2}{dx^2}\right)^{-1} (f'z) \text{ strongly in $L^\infty(\mathbb{R})$}\,.
\end{multline*}
We set $u_\nu=\left(\nu I - \frac{d^2}{dx^2}\right)^{-1} y_{\lambda,\nu}$. Then, for $\nu\to 0$, we have $\nu u_\nu\to 0$ in $L^1(\mathbb{R})$
and
\[
-\lim u_\nu''=\lim y_{\lambda,\nu}=y\quad\text{in $\mathcal{D}'(\mathbb{R})$.}
\]
Hence
\[
 \left(\nu I - \frac{d^2}{dx^2}\right)^{-1} y_{\lambda,\nu} \to  \left(- \frac{d^2}{dx^2}\right)^{-1} y 
\]
strongly in $W^{1,\infty}(\mathbb{R})$, and
\[
 \left(\nu I - \frac{d^2}{dx^2}\right)^{-1} (f y_{\lambda,\nu})\to \left( - \frac{d^2}{dx^2}\right)^{-1} (fy') 
\]
strongly in $L^\infty(\mathbb{R})$,
where $y\in L^1(\mathbb{R})$, and
\[
\lambda\,y- H^*\left(\frac{\sigma^2}{2}y\right)'' -f\,y'=\eta\text{ in $\mathcal{D}'(\mathbb{R})$}\,,
\]
for $\lambda>\lambda_0$.
Moreover, by \eqref{3:12}, the map $\eta\to y$ is Lipschitz in $L^1(\mathbb{R})$, with Lipschitz constant $(\lambda-\lambda_0)^{-1}$, then $y$ solves  \eqref{3:3}, and \eqref{3:4} follows. This completes the proof of Lemma~\ref{lemma3:1}.
\end{proof}

\begin{proof}[Proof of Theorem~\ref{mild:sol} (continued)]
Coming back to equation \eqref{3:1}, by Lemma~\ref{lemma3:1} and \eqref{2:6}, it follows that the operator $A+B$ is quasi-m-accretive in $L^1(\mathbb{R})$. Then by the Crandall \& Ligget theorem, see \cite{B:1}, p. 147, the Cauchy problem \eqref{3:1} has a unique mild solution $y\in C([0,T]; L^1(\mathbb{R}))$, that is 
$$
y(t)=\lim_{\epsilon\to0} y_\epsilon (t)\text{ in $L^1(\mathbb{R})$, $\forall t\in[0,T]$}\,,
$$ \begin{gather*}
y_\epsilon(t)=y^i_\epsilon \text{ for $t\in[i\epsilon,(i+1)\epsilon]$, $i=0,\dots,N=\left[\frac{T}{\epsilon}\right]$}\\
\frac{1}{\epsilon}(y^{i+1}_\epsilon-y^i_\epsilon)+(A+B)(y^{i+1}_\epsilon)=g_1\quad i=0,\dots,N\\
y^0_\epsilon=y_0\,.
\end{gather*}
The function $y$ is a mild solution to \eqref{eq:2} in the sense of Definition~\ref{mild}.

Assume now that $j(\lambda v)\le C_\lambda j(v)$ $\forall v\in\mathbb{R}$ and $\lambda>0$. 
Taking into account that $j(v) \leq j(2v)-v H^*(v), \forall v \in \mathbb{R} \;,$
it is easily seen that this implies that 
\begin{equation}\label{3:23}
H^* (v)v\le \left( C_2-1\right) j(v)\,,\quad\forall v\in\mathbb{R}\,.
\end{equation} 
Assume also that $j(\frac{\sigma^2}{2}\,y_0)\in L^1(\mathbb{R})$. Then, if we take in \eqref{eq:3}, $z^i=\frac{\sigma^2}{2}\,y^i_\epsilon$ and get
\begin{multline*}
\frac{2}{\sigma^2\,\epsilon}\,(z^{i+1}-z^i)-\bigl(H^*(z^{i+1})\bigr)''-f\Bigl(\frac{2}{\sigma^2}\,z^{i+1}\Bigr)'\\+B\Bigl(\frac{2}{\sigma^2}\,z^{i+1}\Bigr)=g_1.
\end{multline*}
Multiplying by $H^*(z^{i+1})$ and integrating on $\mathbb{R}$ we get
\begin{align*}
&\frac{2}{\epsilon}\,\int_\mathbb{R}\frac{1}{\sigma^2}\,\bigl(j(z^{i+1})-j(z^i)\bigr)\,dx+\int_{\mathbb{R}}\Bigl(\bigl(H^*(z^{i+1})\bigr)'\Bigr)^2\,dx\\
&\quad +2\,\int_{\mathbb{R}}f\Bigl(\frac{z^{i+1}}{\sigma^2}\Bigr)\,H^*(z^{i+1})\,dx\\
&\quad+2\,\int_{\mathbb{R}}B\Bigl(\frac{z^{i+1}}{\sigma^2}\Bigr)\,H^*(z^{i+1})\,dx=\int_{\mathbb{R}}g_1\,H^*(z^{i+1})\,dx.
\end{align*}

Integrating by parts in $\int_{\mathbb{R}} f \left( \frac{z^{i+1}}{\sigma^2} \right)' H^*\left(z^{i+1}\right) dy$, summing up, after some calculation involving \eqref{2:6} and \eqref{3:23}, we get the estimate $\forall k$
\[
2\,\int_{\mathbb{R}}\frac{1}{\sigma^2}\,j(z^{k+1})\,dx+\epsilon\,\sum_{i=0}^k\int_{\mathbb{R}}\Bigl(\bigl(H^*(z^{i+1})\bigr)'\Bigr)^2\,dx\le C,
\] 
which implies the desired conclusion
\begin{gather*}
\Bigl(H^*\Bigl(\frac{\sigma^2}{2}\,y\Bigr)\Bigr)_x\in L^2((0,T)\times\mathbb{R}),\\ j\Bigl(\frac{\sigma^2}{2}\,y\Bigr)\in L^\infty([0,T;L^1(\mathbb{R})]).
\end{gather*}
\end{proof}

\section{A multi-dimensional case}\label{sec:nd}
Consider the problem \eqref{min:pb} in $\mathbb{R}^n$ with the drift $f\equiv0$, namely
\begin{equation}\label{4:1}
\text{\underline{Minimize}}\quad
 \mathbb{E}\biggl\{\int_0^Tg\left(X(t)\right)+h\left(u(t)\right)\,dt+g_0\left(X(T)\right)\biggr\}\,,
\end{equation}
subject to $u\in\mathcal{U}$, and to stochastic differential equation
\begin{equation}\label{4:2}
\begin{cases}dX=\sqrt{u}\,\sigma(X)\,dW,\quad\text{in }(0,T)\times\mathbb{R}^n\\
X(0)=X_0\end{cases}\,.
\end{equation}
Here $W\colon[0,T]\to\mathbb{R}^m$ is a Wiener process, $h:\mathbb{R}\to\mathbb{R}$ satisfies assumption (1) and
\begin{enumerate}[(i)]
\item $g,g_0\in W^{2,\infty}(\mathbb{R}^n;\mathbb{R})$
\item $\sigma(x)=\sigma_0(x)a$, where $\sigma_0\in C^1_b(\mathbb{R})$ satisfies condition \eqref{eq:4}, while the matrix $a=\norm{a_{ij}}^{n,m}_{i,j=1}$ is such that $b=aa^T$ is positive defined.
\end{enumerate}
Let $\mathcal{L}$ be the elliptic second order operator
\begin{equation}\label{4:3}
\mathcal{L}z(x)=\sum_{i,j=1}^n b_{ij}\frac{\partial^2 z(x)}{\partial x_i\partial x_j}\,,\quad\forall x\in\mathbb{R}^n
\end{equation}
where $b_{ij}=\sum_{k=1}^m a_{ik}a_{jk}$.
The corresponding dynamic programming equation for \eqref{4:1} reads as follows
\begin{equation}\label{4.4}
\begin{cases}\varphi_t(t,x)+\min_{u}\bigl\{\frac{1}{2}\sigma_0^2(x)\,\mathcal{L} \varphi(t,x)\,u+H(u)\bigr\}\\
\hspace{2.3cm}+g(x)=0,
\quad\forall t\in[0,T],x\in\mathbb{R}^n\\
\varphi(T,x)=g_0(x),\quad x\in\mathbb{R}^n\,.\end{cases}
\end{equation}
If
\begin{equation}\label{4:5}
y(t,x)=-\mathcal{L}\varphi(T-t,x)\,,\quad\forall t\in[0,T],x\in\mathbb{R}^n\,,
\end{equation}
equation \eqref{4:5} reduces to 
\begin{equation}\label{4:6}
\begin{cases}y_t(t,x)-\mathcal{L}\left(H^*\left(\frac{\sigma_0^2(x)}{2}y(t,x)\right)\right)=g_1(x),\\
\hspace{5cm}\forall t\in[0,T],x\in\mathbb{R}^n\\ 
y(0,x)=y_0(x),\quad x\in\mathbb{R}^n,\end{cases}
\end{equation}
see \eqref{DPE:conj:1}, where $y_0=-\mathcal{L}g_0$, $g_1=-\mathcal{L}g$. By \cite{benilan}, for $z\in L^1(\mathbb{R}^n)$ the elliptic equation $-\mathcal{L}\psi=z$ in $\mathcal{D}'(\mathbb{R}^n)$ has a unique solution $\psi$ which satisfies $\psi\in W^{1,\infty}(\mathbb{R})$ if $n=1$, $\psi\in W^{1,1}_\text{loc}(\mathbb{R}^2)$ if $n=2$ and $\psi\in L^1_\text{loc}(\mathbb{R})\cap M^{\frac{n}{n-2}}(\mathbb{R}^n)$ if $n=3$, where here $M^{\frac{n}{n-2}}(\mathbb{R}^n)$ is the Marcinkiewicz space. 
The latter implies that any solution $y\in C([0,T]; L^1(\mathbb{R}^n))$ to \eqref{4:6} leads to a unique solution $\varphi\in C([0,T]; W^{1,\infty}(\mathbb{R}))$ for $n=1$, $\varphi\in C([0,T]; W^{1,1}_\text{loc}(\mathbb{R}^2))$, for $n=2$, and, respectively, $\varphi\in C([0,T];  M^{\frac{n}{n-2}}(\mathbb{R}^n))$ for $n\ge 3$.
Concerning the existence of a solution to eq. \eqref{4:6}, we have a result similar to the one stated in Theorem~\ref{mild:sol},  namely
\begin{theorem}\label{thm4:1} Under assumption (i)-(ii)-(iii) there is a unique mild solution $ y\in C([0,T]; L^1(\mathbb{R}^n))$, in the sense of Definition~\ref{mild}.
\end{theorem}
\begin{proof} We shall proceed as in the proof of Theorem~\ref{mild:sol}. In particular, we consider the operator $A\colon D(A)\subset L^1(\mathbb{R}^n)\to L^1(\mathbb{R}^n)$
\begin{equation}\label{4:7}
Ay=-\mathcal{L}\left(H^*\left(\frac{\sigma_0^2}{2}y\right)\right)\quad\forall y\in D(A)\,,
\end{equation}
$$D(A)=\{y\in L^1(\mathbb{R}^n),\mathcal{L}\left(H^*\left(\frac{\sigma_0^2}{2}y\right)\right)\in L^1(\mathbb{R}^n)\}\,,$$
and we write equation \eqref{4:6} as
\begin{equation}\label{4:8}
\begin{cases}\frac{dy}{dt}+A\,y=g_1,\quad\text{in }[0,T]\\
y(0)=y_0
\end{cases}\end{equation}

\begin{lemma} The operator $A$ is m-accretive in $L^1(\mathbb{R}^n)$. \end{lemma}
\begin{proof} Since the operator $-\Delta$ is m-accretive in $L^1(\mathbb{R}^n)$, see, e.g., \cite{B:2,benilan}, then the same holds for  the operator $-\mathcal{L}$, moreover,  taking into account that $\abs{\sigma_0(x)}\ge\rho>0$, it follows the m-accretivety of the operator $A$, as claimed. Indeed, equation 
\[
\lambda y+ Ay=\eta\text{ in $\mathcal{D}'(\mathbb{R}^n)$}
\]
is equivalent to
\[
\lambda \beta(z)-  \Delta z=\eta \text{ in $\mathcal{D}'(\mathbb{R}^n)$}
\]
where $\beta=\frac{1}{\sigma^2_0}z$ and this implies the conclusion.
\end{proof}

Again invoking the  Crandall \& Ligget Theorem, we get that the eq.  \eqref{4:8} has a unique mild solution $ y\in C([0,T]; L^1(\mathbb{R}^n))$, which is given by
\[
y(t)=\lim_{\epsilon\to0} y_\epsilon (t)\text{ in $L^1(\mathbb{R}^n)$, $\forall t\in[0,T]$}\,,
\]
\begin{gather*}
y_\epsilon(t)=y^i_\epsilon \text{ for $t\in[i\epsilon,(i+1)\epsilon]$, $i=0,\dots,N=\left[\frac{T}{\epsilon}\right]$}\\
\frac{1}{\epsilon}(y^{i+1}_\epsilon-y^i_\epsilon)+A(y^{i+1}_\epsilon)=g_1\,,\quad i=0,\dots,N\\
y^0_\epsilon=y_0 \,,
\end{gather*}
hence completing  the proof of Theorem~\ref{thm4:1}. 
\end{proof}

By Theorem~\ref{thm4:1} it follows the existence and uniqueness of a solution $\varphi\in C([0,T]; L^1_\text{loc}(\mathbb{R}^n)\cap M^{\frac{n}{n-2}}(\mathbb{R}^n))$.

\begin{remark}
In the general $n$-dimensional case, where $f\in C^2_b(\mathbb{R}^n)$, the dynamic programming equation corresponding to \eqref{min:pb} reduces to
\begin{equation}\label{4:11}
\begin{cases}y_t-\mathcal{L}\left(H^*\left(\frac{\sigma_0^2}{2}y\right)\right)- f\cdot\nabla y +B y= \mathcal{L}g_1,\\
\hspace{5cm}\forall t\in[0,T],x\in\mathbb{R}^n\\
y(0)=y_0\,,\quad x\in\mathbb{R}^n,\end{cases}\,,
\end{equation}
where
\begin{multline*}
By=-2\sum_{i,j,k=1}^n b_{ij}D_jf_k\frac{\partial^2}{\partial x_i\partial x_k}\mathcal{L}^{-1}(y) \\- \sum_{i,j,k=1}^n b_{ij}D_{ij}f_k \frac{\partial}{\partial x_k} \mathcal{L}^{-1}(y)
\end{multline*}
therefore eq. \eqref{4:11} can be treated analogously to what we have seen in the 1-dimensional case, at least if the operator $B$ is continuous  in $L^1(\mathbb{R}^n)$, which happens under some  additional conditions on  $f=\{f_k\}_{k=1}^n$. We note that, for $\mathcal{L}=\Delta$, the linear Fokker-Planck equation \eqref{4:11}, has been treated in \cite{B:2}.
\end{remark}

\section{The degenerate 1-D case}

Consider here equation \eqref{eq:2}, that is 
\begin{equation}\label{eq:2:2}
\begin{cases}
y_t-\Bigl(H^*\bigl(\frac{\sigma^2}{2}\,y\bigr)\Bigr)_{xx}-f\,y_x+B\,y=g_1,\quad\text{in }[0,T]\times\mathbb{R}\\
y(0)=y_0\in\mathbb{R}
\end{cases}
\end{equation}
where $\sigma$ is assumed to satisfy the condition $\sigma\in C_b^2(\mathbb{R})$ only.
Moreover, if we consider, as above, the operator $A:D(A)\subset L^1(\mathbb{R})\rightarrow L^1(\mathbb{R})$, such that
\begin{equation}
A\,y=-\Bigl(H^*\Bigl(\frac{\sigma^2}{2}\,y\Bigr)\Bigr)''-f\,y'\,,
\end{equation}
\[
D(A)=\Bigl\{y\in L^1(\mathbb{R})\,;\;f\,y'+\Bigl(H^*\Bigl(\frac{\sigma^2}{2}\,y\Bigr)\Bigr)''\in L^1(\mathbb{R})\Bigr\},
\]
we have the following holds
\begin{lemma}\label{lemma5:1}
$A$ is quasi-$m$-accretive in $L^1(\mathbb{R})$.
\end{lemma}

\begin{proof}
For each $\epsilon>0$ we consider the operator
\begin{equation}
A_\epsilon\,y=-\Bigl(H^*\Bigl(\frac{\sigma^2+\epsilon}{2}\,y\Bigr)\Bigr)''-f\,y'\,,
\end{equation}
which is quasi-$m$-accretive, seen Lemma \ref{lemma3:1}.
Hence, for each $\eta\in L^1(\mathbb{R})$ and $\lambda\ge\lambda_0$ the equation
\begin{equation}\label{5:4}
\lambda\,y_\epsilon-\Bigl(H^*\Bigl(\frac{\sigma^2+\epsilon}{2}\,y_\epsilon\Bigr)\Bigr)''-f\,y_\epsilon'=\eta,\quad\text{in }\mathbb{R},
\end{equation}
has a unique solution $y_\epsilon\in L^1(\mathbb{R})$, with $H^*\Bigl(\frac{\sigma^2+\epsilon}{2}\,y_\epsilon\Bigr)\in L^\infty(\mathbb{R})$.

\underline{Dynamic estimates.} As in the proof of Lemma \ref{lemma3:1}, we have
\begin{equation}
\lambda\,\norm{y_\epsilon}_1\le\norm{\eta}_1+\norm{f'}_1\,\norm{y_\epsilon}_1,\quad\forall\epsilon>0\,,
\end{equation} that is for $\lambda>\norm{f'}_\infty$
\begin{equation}
\norm{y_\epsilon}_1\le(\lambda-\norm{f'}_1)^{-1}\,\norm{\eta}_1,\quad\forall\epsilon>0\,.
\end{equation}
Assume now that $\eta\in L^1(\mathbb{R})\cap L^{\infty}(\mathbb{R})$, then, by \eqref{5:4} we see that for each $M>0$
\begin{align*}
&\lambda\,(y_\epsilon-M)-\Bigl(H^*\Bigl(\frac{\sigma^2+\epsilon}{2}\,y_\epsilon\Bigr)-H^*\Bigl(\frac{\sigma^2+\epsilon}{2}\,M\Bigr)\Bigr)''\\
&\quad-f(y_\epsilon-M)'=\eta-\lambda\,M+\Bigl(H^*\Bigl(\frac{\sigma^2+\epsilon}{2}\,M\Bigr)\Bigr)''=\tilde{\eta}.
\end{align*}
Moreover, by \eqref{eq:5prime}, we also have 
\begin{multline*}
\tilde{\eta}(x)\leq \eta-M\lambda+M^2\| (H^*)''\|_{\infty} \| \sigma \sigma' \|_{\infty}+\\
+M \|(H^*)'\|_{\infty} \| \sigma \sigma'' +(\sigma')^2\|_{\infty} \leq 0
\end{multline*}
for $M$ and $\lambda$ large enough (independently of $\epsilon$).
This yields
\[
\lambda\,\norm{(y_\epsilon-M)^+}_{1}\le\norm{f'}_\infty
\,\norm{(y_\epsilon-M)^+}_1.
\]
Hence $y_\epsilon\le M$ in $\mathbb{R}$ for $\lambda>\norm{f'}_\infty$.
Similarly, it follows that
\begin{align*}
&\lambda\,(y_\epsilon+M)-\Bigl(H^*\Bigl(\frac{\sigma^2+\epsilon}{2}\,y_\epsilon\Bigr)-H^*\Bigl(-\frac{\sigma^2+\epsilon}{2}\,M\Bigr)\Bigr)''+\\
&\hspace{2cm}+f(y_\epsilon+M)'=\\
&\hspace{3cm}=\eta+\lambda\,M+\Bigl(H^*\Bigl(-\frac{\sigma^2+\epsilon}{2}\,M\Bigr)\Bigr)''\\
&\hspace{3cm}=\eta+\lambda\,M\ge0,
\end{align*}
if $M$ is large enough, but independent of $\epsilon$. Therefore, if multiply the equation by $(y_\epsilon+M)^-$ and integrate on $\mathbb{R}$, we get $\norm{(y_\epsilon+M)^-}_1\ge0$ which implies $y_\epsilon\ge-M$ in $\mathbb{R}$.

By \eqref{5:4}, we see that
$\Bigl\{\Bigl(H^*\Bigl(\frac{\sigma^2+\epsilon}{2}\,y_\epsilon\Bigr)\Bigr)'+f\,y_\epsilon'\Bigr\}_{\epsilon>0}$ is bounded in $W^{1,\infty}(\mathbb{R})$.

Hence $\Bigl(H^*\Bigl(\frac{\sigma^2+\epsilon}{2}\,y_\epsilon\Bigr)\Bigr)'$ bounded in $L^1(\mathbb{R})\cap L^\infty(\mathbb{R})$, so that
$\Bigl\{\eta_\epsilon=H^*\Bigl(\frac{\sigma^2+\epsilon}{2}\,y_\epsilon\Bigr)\Bigr\}$ is compact in $C(\mathbb{R})$.
It follows that on a subsequence $\epsilon\rightarrow 0$, we have
\begin{align*}
&y_\epsilon\rightarrow y,\quad\text{weakly in all }L^p,1<p\le\infty\,,\\
&\eta_\epsilon\rightarrow\zeta,\quad\text{strongly in }C(\mathbb{R})\,,
\end{align*}
where $\zeta=H^*\Bigl(\frac{\sigma^2}{2}\,y\Bigr)$ in $\mathbb{R}$.
Letting $\epsilon\rightarrow 0$ in \eqref{5:4}, we get
\[
\lambda\,y-\Bigl(H^*\Bigl(\frac{\sigma^2}{2}\,y\Bigr)\Bigr)''+f\,y'=\eta,\quad\text{in }\mathcal{D}'(\mathbb{R}).
\]
Next for $\eta\in L^1(\mathbb{R})$ we choose $\{\eta_n\}\subset L^1(\mathbb{R})\cap L^\infty(\mathbb{R})$, $\eta_n\rightarrow\eta$ in $L^1(\mathbb{R})$ and we have
\[
\lambda\,y_n-\Bigl(H^*\Bigl(\frac{\sigma^2}{2}\,y_n\Bigr)\Bigr)''+f\,y_n'=\eta_n\,,\quad\forall n\,,
\]
 getting \[
\lambda\norm{y_n-y_m}_1\le\norm{\eta_n-\eta_m}_1+\norm{f'}_\infty\norm{y_n-y_m}_1,\,\forall n,m.
\]
Hence, for $\lambda>\norm{f'}_\infty$ we have for $n\rightarrow\infty$
\begin{align*}
&y_n\rightarrow y,\quad \text{strongly in }L^1(\mathbb{R})\\
&\Bigl(H^*\Bigl(\frac{\sigma^2}{2}\,y_n\Bigr)\Bigr)\rightarrow\Bigl(H^*\Bigl(\frac{\sigma^2}{2}\,y\Bigr)\Bigr),\quad\text{a.e. in }\mathbb{R}\\
&f\,y_n'\rightarrow f\,y',\quad\text{in }\mathcal{D}'(\mathbb{R})\,.
\end{align*}
This yields 
\[
\lambda\,y-\Bigl(H\Bigl(\frac{\sigma^2}{2}\,y\Bigr)\Bigr)''-f\,y'=\eta,\quad\text{in }\mathcal{D}'(\mathbb{R}).
\]
Hence for $\lambda\ge\lambda_0$, $y\in L^1(\mathbb{R})$ is the solution to equation $\lambda\,y+A\,y=\eta$ as claimed.
As seen earlier this implies that the operator $A+B$ is quasi-$m$-accretive in $L^1(\mathbb{R})$
\end{proof}

Then by the existence theorem for the equation
\begin{equation*}
\begin{cases}
\frac{\partial y}{\partial t}+A\,y+B\,y=0\\
y(0)=y_0,
\end{cases}
\end{equation*}
we get

\begin{theorem}\label{Thm5:1}
There is a unique mild solution $y\in C([0,T];\mathbb{R})$ to equation \eqref{eq:2:2}.
\end{theorem}

As in previous case Theorem \ref{Thm5:1} implies via \eqref{2:4}  the existence of a mild solution $\varphi$ to equation \eqref{min:pb} satisfying \eqref{2:11}. We omit the details.

\section{Conclusions}
In this paper it is shown, via nonlinear semigroup theory in $L^1$, both the existence and the uniqueness of a mild solution for the dynamic programming equation for stochastic optimal control problem with control in the volatility term. Latter problem is related to the analysis of controlled stochastic volatility models, within the financial frameworks, whose related computational study  is the subject of our ongoing research.

\bibliographystyle{plain}        
\bibliography{autosam}           

\begin{thebibliography}{10}

\bibitem{B:1}
Viorel Barbu.
\newblock {\em Nonlinear differential equations of monotone types in Banach
  spaces}.
\newblock Springer Science \& Business Media, 2010.

\bibitem{B:2}
Viorel Barbu.
\newblock Generalized solutions to nonlinear fokker--planck equations.
\newblock {\em Journal of Differential Equations}, 261(4):2446--2471, 2016.

\bibitem{benilan}
Philippe Benilan, Haim Brezis, and Michael~G Crandall.
\newblock A semilinear equation in $ l^1(\mathbb{R}^n)$.
\newblock {\em Annali della Scuola Normale Superiore di Pisa-Classe di
  Scienze}, 2(4):523--555, 1975.

\bibitem{CDP}
Francesco. Cordoni and Luca Di~Persio.
\newblock Transition density for cir process by lie symmetries and application
  to zcb pricing.
\newblock {\em International Journal of Pure and Applied Mathematics},
  88(2):239--246, 2013.

\bibitem{CEV}
John~C. Cox.
\newblock Notes on option pricing i: Constant elasticity of diffusions.
\newblock {\em Stanford University}, Unpublished draft(2), 1975.

\bibitem{crandall}
Michael~G Crandall, Hitoshi Ishii, and Pierre-Louis Lions.
\newblock User’s guide to viscosity solutions of second order partial
  differential equations.
\newblock {\em Bulletin of the American Mathematical Society}, 27(1):1--67,
  1992.

\bibitem{fleming}
Wendell~H Fleming and Raymond~W Rishel.
\newblock {\em Deterministic and stochastic optimal control}, volume~1.
\newblock Springer Science \& Business Media, 2012.

\bibitem{Hagan}
P.~Hagan, A.~Lesniewski, and D.~Woodward.
\newblock Probability distribution in the sabr model of stochastic volatility.
\newblock volume 110, pages 1--35, 2015.

\bibitem{Heston}
Steven~L. Heston.
\newblock A closed-form solution for options with stochastic volatility with
  applications to bond and currency options.
\newblock {\em The Review of Financial Studies}, 6(2):327--343, 1993.

\bibitem{CIR}
Jr. John C.~Cox, Jonathan E.~Ingersoll and Stephen~A. Ross.
\newblock An intertemporal general equilibrium model of asset prices.
\newblock {\em Econometrica}, 53(2):363--384, 1985.

\bibitem{oksendal}
Bernt {\O}ksendal.
\newblock Stochastic differential equations.
\newblock In {\em Stochastic differential equations}, pages 65--84. Springer,
  2003.

\bibitem{Stein1}
Jerome~L. Stein.
\newblock {\em Stochastic Optimal Control, International Finance, and Debt
  Crises}, volume~1.
\newblock OUP Oxford, 2006.

\bibitem{Stein2}
Jerome~L. Stein.
\newblock {\em Stochastic Optimal Control and the U.S. Financial Debt Crisis},
  volume~1.
\newblock Springer Science and Business Media, 2012.

\end{thebibliography}
\end{document}